\documentclass[a5paper, 10pt]{article}

\usepackage{cite, amsmath, amssymb, booktabs, lastpage, graphicx}
\usepackage[hidelinks]{hyperref}

\usepackage{geometry}
\usepackage{setspace}

\usepackage{amsthm}

\theoremstyle{plain}

\newtheorem{theorem}{Theorem}
\newtheorem{conjecture}{Conjecture}

\newtheorem{corollary}{Corollary}
\newtheorem{claim}{Claim}
\newtheorem{observation}{Observation}

\theoremstyle{remark}

\theoremstyle{definition}

\usepackage{graphicx}

\newcommand{\address}[1]{\def\@address{#1}}
\newcommand{\email}[1]{\def\@email{#1}}

\makeatother

\newcommand{\acknowledgment}[1]{\vspace{5mm}\singlespacing
	{\noindent\textbf{\textit{Acknowledgment\/}:} #1}
}

\usepackage{fancyhdr}

\usepackage[margin=1cm,%
			font=footnotesize,%
			format=hang,%
			labelsep=period,%
			labelfont=bf]{caption}

\title{Extremal graphs with maximum complementary second Zagreb index}
\author{Hui Gao}

\address{School of Mathematics, China University of Mining and Technology, Xuzhou, Jiangsu 221116, China }

\email{gaoh1118@yeah.net}

\date{\today}

\begin{document}

\maketitle

\begin{abstract}
Recently, a couple of degree-based topological indices, defined using a geometrical point of view of a graph edge, have attracted significant attention and being extensively investigated. Furtula and Oz [Complementary Topological Indices, \textit{MATCH Commun. Math. Comput. Chem.\/}
	\textbf{93} (2025) 247--263] introduced a novel approach for devising ``geometrical'' topological indices and focused special attention on the complementary second Zagreb index as a representation of the introduced approach. In the same paper, they also conjectured the maximal graphs  of order $n$ with the maximum complementary second Zagreb index. In this paper, we confirm their conjecture.
\end{abstract}

\onehalfspacing

\section{Introduction}


All graphs, digraphs and mixed graphs in this paper are considered to be simple, that is they have no loops or parallel edges or arcs.

Let $F$ be a mixed graph with vertex set $V(F)$, edge set $E(F)$ and arc set $A(F)$. Since a graph or digraph can be seen as a special mixed graph with empty arc set or empty edge set respectively, the following definitions are claimed only for mixed graphs.

For $X \subseteq V(F)$, denote by $F[X]$ the mixed graph with vertex set $X$ and edge set and arc set consisting of edges and arcs in $F$ with both of their end-vertices in $X$ respectively. For $u \in V(F)$, denote by $d_{F}^{+}(u)$, $d_{F}^{-}(u)$ and $d_{F}(u)$ the number of arcs with $u$ as their tails, the number of arcs with $u$ as their heads and the total number of edges and arcs incident with $u$, respectively.
For an edge set $E_{0}$ and arc set $A_{0}$, denote by $ F\pm E_{0} $ ($F \pm A_{0}$)  the mixed graph by adding  to $F$ or deleting in $F$ all edges in $E_{0}$ (all arcs in $A_{0}$) respectively.
Let $G$ and $H$ be graphs. Denote by $\overline{G}$ the graph with vertex set $V(G)$ and edge set consisting of  pairs of nonadjacent vertices of $G$. Denote by $G \vee H$ the graph by starting with a disjoint union of two graphs $G$ and $H$ and adding edges joining every vertex of $G$ to every vertex of $H$.

Chemical  Graph  Theory  is  a  branch  of  Mathematical  Chemistry  which  has  an important  effect  on  the  development  of  Chemical  Sciences.  A  molecular  graph  or chemical  graph  is  a  simple  graph  such  that  its  vertices correspond  to the atoms and the edges  to  the  bonds.  In  Chemistry,  degree-based topological  indices  have  been  found  to  be  useful  in discrimination,   chemical   documentation,   structure   property   relationships,   structure   activity  relationships  and  pharmaceutical  drug  design.  There  has  been  considerable interest in the general problem of determining degree-based topological indices, see \cite{GP-86,K-18}.

Furtula and Oz \cite{FO-25} introduced a novel way to construct the indices using complement degree points: by substituting end-vertex degrees $d_{G}(u)$ and $d_{G}(v)$ with $d_{G}(u)+d_{G}(v)$ and $d_{G}(u)-d_{G}(v)$ in the definitions of the existing degree-based topological indices. Using this approach they got a whole new group of degree topological descriptors named as complementary topological indices.  Special attention was focused on the existing second Zagreb index $M_{2}(G)$ of a graph $G$, where
$$M_{2}(G):= \sum_{uv \in E(G)} d_{G}(u) d_{G}(v).$$
As a representation of the introduced approach, with a slight rectification, they  introduced the complementary second Zagreb index $cM_{2}(G)$:

$$cM_{2}(G)= \sum_{uv \in E(G)} |d_{G}(u)^{2}-d_{G}(v)^{2}|.$$

However, this index is not put forward here for the first time. It was introduced and reintroduced in several recent and unrelated papers, which
resulted in several names for this index, such as nano Zagreb index~\cite{JS-19}, minus-$F$ index~\cite{K-19}, modified Albertson index~\cite{YBA-20} and first Sombor index~\cite{G-22,ILJAF-22,TLD-23}.
Later, this index was investigated for some supramolecular chains~\cite{ILJAF-22}, and maximal trees and unicyclic graphs with some given parameters were determined in~\cite{L-23}. Additionally, in \cite{L-23} the correlations of this index with some physicochemical properties of octanes and benzenoid hydrocarbons were displayed.

Also, Furtula and Oz \cite{FO-25} conjectured the structure of the maximal graphs  of order $n$ with the maximum complementary second Zagreb index.

\begin{conjecture}[\cite{FO-25}]\label{Max-cM2}
The graph with $n$ vertices and the maximum complementary second Zagreb index is isomorphic to $K_{m} \vee \overline{K_{n-m}}$ for some $1 \leq m \leq n-1$.
\end{conjecture}

In this paper, we confirm their conjecture.

\begin{theorem}\label{maintheorem}
Conjecture~\ref{Max-cM2} is true.
\end{theorem}

\section{Preliminaries}

Let $G=(V, E)$ be an undirected graph with $n$ vertices and $F$ be the mixed graph obtained by orienting some edges in $G$ as follows:

\noindent (i) $ \overrightarrow{uv} \in A(F) \text{~if~} uv \in E(G) \text{~and~} d_{G}(u)>d_{G}(v),\\$
(ii) $ uv \in E(F)  \text{~if~} uv \in E(G) \text{~and~} d_{G}(u)=d_{G}(v).$

\noindent Denote by $X$ and $Y$ the set of vertices satisfying $ d_{F}^{+}(u) \geq d_{F}^{-}(u)$ and $ d_{F}^{+}(u) < d_{F}^{-}(u)$ respectively.

In this section, we will show some basic properties of $cM_{2}(G)$ and graph operations that will increase $cM_{2}(G)$.

\begin{observation}\label{OB-1}
For any mixed graph $F'$ obtained by orienting some edges in $G$, we have
$$cM_{2}(G)\geq \sum_{\overrightarrow{uv} \in A(F')} (d_{F'}(u)^{2}-d_{F'}(v)^{2}).$$
If the mixed graph  $F'=F$, the equality holds.
\end{observation}
\begin{observation}\label{OB-2}
For any mixed graph $F'$ obtained by orienting some edges in $G$, we have
\[
\begin{split}
cM_{2}(G)
&\geq \sum_{\overrightarrow{uv} \in A(F')} (d_{F'}(u)^{2}-d_{F'}(v)^{2}) ~~\text{(by Observation~\ref{OB-1})}\\
&= \sum_{v \in V(G)}(-\sum_{\overrightarrow{uv} \in A(F')}  d_{F'}(v)^{2} + \sum_{\overrightarrow{vw} \in A(F')}d_{F'}(v)^{2})\\
&= \sum_{v \in V(G)}(d_{F'}^{+}(v) - d_{F'}^{-}(v))d_{F'}(v)^{2}.
\end{split}
\]
If the mixed graph  $F'=F$, the equality holds.
\end{observation}

\noindent \textbf{Operation A:} For any $u, v \in X$, if $uv \notin E(G)$, then $cM_{2}(G+uv) \geq cM_{2}(G)$. If the equality holds, then $d_{F}(u)=d_{F}(v)$, $d_{F}^{+}(u)=d_{F}^{-}(u)$ and $d_{F}^{+}(v)=d_{F}^{-}(v)$.

\begin{proof}
Without loss of generality, suppose $d_{G}(u) \geq d_{G}(v)$. Add $\overrightarrow{uv}$ to $F$. Then for any vertex $v'$ in $G$ except $u$ and $v$, the degree, out-degree and in-degree of $v'$ in $F+\overrightarrow{uv}$ is the same as in $F$. Note that since $u, v \in X$, $d_{F}^{+}(u)\geq d_{F}^{-}(u)$ and $d_{F}^{+}(v)\geq d_{F}^{-}(v)$. Then we have
\[
\begin{split}
& cM_{2}(G+uv)-cM_{2}(G)\\
& \geq \sum_{v' \in V(G)}(d_{F+\overrightarrow{uv}}^{+}(v')-d_{F+\overrightarrow{uv}}^{-}(v'))d_{F+\overrightarrow{uv}}(v')^{2}\\
& ~~~-\sum_{v' \in V(G)}(d_{F}^{+}(v')-d_{F}^{-}(v'))d_{F}(v')^{2}~~~\text{(by Observation\ref{OB-2})}\\
& =(d_{F+\overrightarrow{uv}}^{+}(u)-d_{F+\overrightarrow{uv}}^{-}(u))d_{F+\overrightarrow{uv}}(u)^{2}+
(d_{F+\overrightarrow{uv}}^{+}(v)-d_{F+\overrightarrow{uv}}^{-}(v))d_{F+\overrightarrow{uv}}(v)^{2}\\
& ~~~-( (d_{F}^{+}(u)-d_{F}^{-}(u))d_{F}(u)^{2}+(d_{F}^{+}(v)-d_{F}^{-}(v))d_{F}(v)^{2} )\\
& =(d_{F}^{+}(u)+1-d_{F}^{-}(u))(d_{F}(u)+1)^{2}+(d_{F}^{+}(v)-d_{F}^{-}(v)-1)(d_{F}(v)+1)^{2}\\
& ~~~-( (d_{F}^{+}(u)-d_{F}^{-}(u))d_{F}(u)^{2}+(d_{F}^{+}(v)-d_{F}^{-}(v))d_{F}(v)^{2} )\\
& = (d_{F}^{+}(u)-d_{F}^{-}(u))(2d_{F}(u)+1)+(d_{F}^{+}(v)-d_{F}^{-}(v))(2d_{F}(v)+1)\\
& ~~~+(d_{F}(u)+1)^{2}-(d_{F}(v)+1)^{2}\\
& \geq 0~~\text{(since $d_{F}^{+}(u) \geq d_{F}^{-}(u)$, $d_{F}^{+}(v) \geq d_{F}^{-}(v)$ by  $u,v \in X$ and $d_{G}(u) \geq d_{G}(v)$)}.
\end{split}
\]
If $cM_{2}(G+uv)=cM_{2}(G)$, then the second ``$\geq$'' should be ``$=$'', which induces that $d_{F}(u)=d_{F}(v)$, $d_{F}^{+}(u)=d_{F}^{-}(u)$ and $d_{F}^{+}(v)=d_{F}^{-}(v)$.
\end{proof}

\noindent \textbf{Operation B:} For any $u, w \in V, v \in Y$, if $\overrightarrow{uv}, \overrightarrow{vw} \in A(F)$, $uw \notin E(G)$, then $cM_{2}(G-uv-vw+uw) > cM_{2}(G)$.

\begin{proof}
Notice that $G-uv-vw+uw $ is the underlying graph of the mixed graph $F-\overrightarrow{uv}-\overrightarrow{vw}+\overrightarrow{uw}$, and for any vertex $v'$ in $G$ except $v$, the degree, out-degree and in-degree of $v'$ in $F-\overrightarrow{uv}-\overrightarrow{vw}+\overrightarrow{uw}$ is same as in $F$. Hence we have
\[
\begin{split}
& cM_{2}(G-uv-vw+uw)-cM_{2}(G)\\
& \geq  (d_{F-\overrightarrow{uv}-\overrightarrow{vw}+\overrightarrow{uw}}^{+}(v)-d_{F-\overrightarrow{uv}-\overrightarrow{vw}+\overrightarrow{uw}}^{-}(v))
d_{F-\overrightarrow{uv}-\overrightarrow{vw}+\overrightarrow{uw}}(v)^{2}\\
& ~~~-(d_{F}^{+}(v)-d_{F}^{-}(v))d_{F}(v)^{2}~~~\text{(by Observation~\ref{OB-2}) }\\
& =((d_{F}^{+}(v)-1)-(d_{F}^{-}(v)-1))(d_{F}(v)-2)^{2}-(d_{F}^{+}(v)-d_{F}^{-}(v))d_{F}(v)^{2}\\
& =-4(d_{F}^{+}(v)-d_{F}^{-}(v))(d_{F}(v)-1)\\
& >0~~~\text{(since $d_{F}^{+}(v)<d_{F}^{-}(v)$ by $v \in Y$)}.
\end{split}
\]
\end{proof}

\section{Proof of Theorem~\ref{maintheorem}}

Suppose $G=(V, E)$ is an undirected graph with $n$ vertices and maximum complementary second Zagreb index. Let $F$ be the mixed graph as the last section stated. Then we will show that $G$ is isomorphic to $K_{m}\vee \overline{K_{n-m}}$ for some $1 \leq m \leq n-1$.

\begin{claim}\label{claim1}
For any $u, v \in X$, if $uv \notin E(G)$, then $d_{G}(u)=d_{G}(v)$ and  for any $w \in X \setminus \{u, v\}$, $d_{G}(w) \neq d_{G}(u)$.
\end{claim}

\begin{proof}
By Operation A, $cM_{2}(G+uv) \geq cM_{2}(G)$. Since $G$ has maximum complementary second Zagreb index, $cM_{2}(G+uv) = cM_{2}(G)$. Hence, by Operation A, $d_{G}(u) = d_{G}(v)$, $d_{F}^{+}(u)=d_{F}^{-}(u)$ and $d_{F}^{+}(v)=d_{F}^{-}(v)$.

Suppose to the contrary  that there exists $w \in X \setminus \{u,v\}$ such that $d_{G}(w)=d_{G}(u)$.

\textbf{Case 1:}If $uw \neq E(G)$, then by Operation A, $d_{F}^{+}(w)=d_{F}^{-}(w)$. Add $\overrightarrow{uv}$ and $\overrightarrow{uw}$ to $F$. Then we have
\[
\begin{split}
& cM_{2}(G+uv+uw)-cM_{2}(G)\\
& \geq \sum_{v'=u, v, w} (d_{F+ \overrightarrow{uv} +\overrightarrow{uw}}^{+}(v')-d_{F+ \overrightarrow{uv} +\overrightarrow{uw}}^{-}(v'))d_{F+ \overrightarrow{uv} +\overrightarrow{uw}}(v')^{2}\\
& ~~~-\sum_{v'=u, v, w} (d_{F}^{+}(v')-d_{F}^{-}(v'))d_{F}(v')^{2}~~~\text{(by Observation~\ref{OB-2})}\\
& =2(d_{F}(u)+2)^{2}-(d_{F}(v)+1)^{2}-(d_{F}(w)+1)^{2}\\
& ~~~\text{(since $d_{F}^{+}(u)=d_{F}^{-}(u)$ ,$d_{F}^{+}(v)=d_{F}^{-}(v)$ and $d_{F}^{+}(w)=d_{F}^{-}(w)$)}\\
& >0,
\end{split}
\]
contradicting that $G$ has maximum complementary second Zagreb index.

\textbf{Case 2:} If $uw \in E(G)$, since $d_{G}(u)=d_{G}(w)$, $uw \in E(F)$. Consider the mixed graph $F-uw+\overrightarrow{uw}+\overrightarrow{uv}$, whose underlying graph is $G+uv$.
\[
\begin{split}
& cM_{2}(G+uv)-cM_{2}(G)\\
& \geq \sum_{v'=u, v, w} (d_{F-uw+\overrightarrow{uw}+\overrightarrow{uv}}^{+}(v')-d_{F-uw+\overrightarrow{uw}+\overrightarrow{uv}}^{-}(v'))
d_{F-uw+\overrightarrow{uw}+\overrightarrow{uv}}(v')^{2}\\
& ~~~-\sum_{v'=u, v, w} (d_{F}^{+}(v')-d_{F}^{-}(v'))d_{F}(v')^{2}~~~\text{(by Observation~\ref{OB-2})}\\
& \geq 2(d_{F}(u)+1)^{2}-(d_{F}(v)+1)^{2}+(d_{F}^{+}(w)-d_{F}^{-}(w)-1)d_{F}(w)^{2}\\
& ~~~-(d_{F}^{+}(w)-d_{F}^{-}(w))d_{F}(w)^{2}~~~\text{(since $d_{F}^{+}(u)=d_{F}^{-}(u)$ and $d_{F}^{+}(v)=d_{F}^{-}(v)$) } \\
& =2(d_{F}(u)+1)^{2}-(d_{F}(v)+1)^{2}-d_{F}(w)^{2}\\
& >0~~~\text{(since $d_{G}(u)=d_{G}(v)=d_{G}(w))$},
\end{split}
\]
contradicting that $G$ has maximum complementary second Zagreb index.

Above all, for any $w \in X$, $d_{G}(w) \neq d_{G}(u)$.
\end{proof}

\begin{corollary}\label{cor-1}
For any $u \in X$, $d_{G[x]}(u) \geq |X|-2$.
\end{corollary}

\begin{proof}
Let $u \in X$. Suppose to the contrary that $d_{G[X]}(u) < |X|-2$. Then there exists $v \in X$ such that $uv \notin E(G)$. By Claim~\ref{claim1}, for any $w \in X\setminus \{u, v\}$, $d_{G}(w)\neq d_{G}(u)$; thus also by Claim~\ref{claim1}, we have $uw \in E(G)$. So $d_{G[X]}(u)=|X|-2$, a contradiction.
\end{proof}

\begin{claim}\label{claim2}
For any $u, v \in Y$, if $uv \in E(G)$, then $uv \notin E(F)$.
\end{claim}
\begin{proof}
Suppose to the contrary that $uv \in E(F)$ for some $u, v \in Y$. Note that since $u ,v \in Y$, $d_{F}^{+}(u)<d_{F}^{-}(u)$ and $d_{F}^{+}(v)<d_{F}^{-}(v)$. Then
\[
\begin{split}
& cM_{2}(G-uv)-cM_{2}(G)\\
& \geq (d_{F-uv}^{+}(u)-d_{F-uv}^{-}(u))d_{F-uv}(u)^{2}+(d_{F-uv}^{+}(v)-d_{F-uv}^{-}(v))d_{F-uv}(v)^{2}\\
& ~~~-( (d_{F}^{+}(u)-d_{F}^{-}(u))d_{F}(u)^{2}+(d_{F}^{+}(v)-d_{F}^{-}(v))d_{F}(v)^{2})~~~\text{(by Observation~\ref{OB-2})}\\
& =-( (d_{F}^{+}(u)-d_{F}^{-}(u))(2d_{F}(u)-1)-( (d_{F}^{+}(v)-d_{F}^{-}(v))(2d_{F}(v)-1)\\
& >0 ~~~\text{(since $d_{F}^{+}(u)<d_{F}^{-}(u)$ and $d_{F}^{+}(v)<d_{F}^{-}(v)$ by $u ,v \in Y$)},
\end{split}
\]
contradicting that $G$ has maximum complementary second Zagreb index.
\end{proof}

\begin{claim}\label{claim3}
Suppose $v \in Y$ and $ N_{F}^{-}(v) \subseteq X$. Then for any $u \in X$, if $uv \in E(G)$, then $\overrightarrow{uv} \in A(F)$.
\end{claim}
\begin{proof}
Suppose to the contrary that $uv \in E(F)$ or $\overrightarrow{vu} \in A(F)$ for some $u \in X$.

For any $w  \in N_{F}^{-}(v)$, that is $\overrightarrow{wv} \in A(F)$, we can show that $\overrightarrow{wu} \in A(F)$, and thus $d^{-}_{F}(u) \geq d^{-}_{F}(v)$. In fact, since $\overrightarrow{wv} \in A(F)$, we have $d_{G}(w) > d_{G}(v)$; since $uv \in E(F)$ or $\overrightarrow{vu} \in A(F)$, we have $ d_{G}(v) \geq d_{G}(u) $ and thus $d_{G}(w) > d_{G}(u)$. Since $ N_{F}^{-}(v) \subseteq X$, we have $w \in X$; since $d_{G}(w) > d_{G}(u)$, by Claim~\ref{claim1}, we have $wu \in E(G)$ and thus $\overrightarrow{wu} \in A(F)$.

Since $u \in X$ and $v \in Y$, we have $ d^{+}_{F}(u) \geq d^{-}_{F}(u)$ and $d^{+}_{F}(v) < d^{-}_{F}(v) $. Note that $d^{-}_{F}(u) \geq d^{-}_{F}(v)$. Hence,  $d^{+}_{F}(u) \geq d^{-}_{F}(u)\geq  d^{-}_{F}(v) > d^{+}_{F}(v)$ and thus $d^{+}_{F}(u) + d^{-}_{F}(u) >  d^{-}_{F}(v) +d^{+}_{F}(v)$.

\textbf{Case 1: } $uv \in E(F)$.

Since $uv \in E(F)$, we have  $d_{F}(u)=d_{F}(v)$ and thus there exists $x \in V(G)$ such that $xv \in E(F)$ and $xu \notin E(F)$. Since $v \in Y$, by Claim~\ref{claim2}, we have $x \in X$; since $xu \notin E(F)$, or equivalently to say $xu \notin E(G)$, by Claim~\ref{claim1}, we have $d_{G}(x)=d_{G}(u)$.
\[
\begin{split}
& cM_{2}(G+ux)-cM_{2}(G)\\
& \geq \sum_{v'=u, x, v} (d_{F-uv+\overrightarrow{uv}+\overrightarrow{ux}}^{+}(v')-d_{F-uv+\overrightarrow{uv}+\overrightarrow{ux}}^{-}(v'))
d_{F-uv+\overrightarrow{uv}+\overrightarrow{ux}}(v')^{2}\\
& ~~~-\sum_{v'=u, x, v} (d_{F}^{+}(v')-d_{F}^{-}(v'))d_{F}(v')^{2}~~~\text{(by Observation~\ref{OB-2})}\\
& = 2(d_{F}(u)+1)^{2}-(d_{F}(x)+1)^{2}+(d_{F}^{+}(v)-d_{F}^{-}(v)-1)d_{F}(v)^{2}\\
& ~~~-(d_{F}^{+}(v)-d_{F}^{-}(v))d_{F}(v)^{2}\\
& ~~~\text{(since $d_{F}^{+}(u)=d_{F}^{-}(u)$ and $d_{F}^{+}(x)=d_{F}^{-}(x)$ by Operation A) } \\
& =2(d_{F}(u)+1)^{2}-(d_{F}(x)+1)^{2}-d_{F}(v)^{2}\\
& >0~~~\text{(since $d_{G}(u)=d_{G}(v)=d_{G}(x))$},
\end{split}
\]
contradicting that $G$ has maximum complementary second Zagreb index.

\textbf{Case 2:} $\overrightarrow{vu} \in A(F)$.

By Case 1, we know that there exists no vertex $w \in X$ such that $wv \in E(F)$; by Claim~\ref{claim2}, there exists no vertex $w \in Y$ such that $wv \in E(F)$. Hence, there exists no edge in $E(F)$ incident with $v$ and thus $d_{F}(v)=d_{F}^{+}(v)+d_{F}^{-}(v)$. Note that $d^{+}_{F}(u) + d^{-}_{F}(u) >  d^{-}_{F}(v) +d^{+}_{F}(v)$. Hence, $d_{F}(u)>d_{F}(v)$. However, since $\overrightarrow{vu} \in A(F)$, we have $d_{F}(v) >d_{F}(u)$, a contradiction.
\end{proof}

\begin{claim}\label{claim4}
$G[Y]$ is empty.
\end{claim}
\begin{proof}
Suppose to the contrary that $G[Y]$ is not empty. By Claim~\ref{claim2}, $F[Y]$ is a nonempty digraph. Choose a vertex $u \in Y$ such that $d^{+}_{F[Y]}(u)>0$ and $d_{F}(u)$ is maximum. Then we can show that $N_{F}^{-}(u) \subseteq X$. Suppose to the contrary that there exists $w \in Y$ such that $\overrightarrow{wu} \in A(F)$. Then $ d_{F}(w) > d_{F}(u)$ and $d_{F[Y]}^{+}(w) >0$, a contradiction to the choice of $u$.

Since $N_{F}^{-}(u) \subseteq X$, by Claim~\ref{claim3}, we have $N_{F}^{+}(u) \subseteq Y$. Suppose $N_{F}^{+}(u)$ consists of $v_{1}, v_{2}, \ldots, v_{t}$. Denote $E:=\{uv_{1}, uv_{2}, \ldots, uv_{t} \}$ and $ \overrightarrow{E}:=\{\overrightarrow{uv_{1}}, \overrightarrow{uv_{2}}, \ldots, \overrightarrow{uv_{t}} \}$. For any $x \in N_{F}^{-}(u)$ and $v_{i} \in N_{F}^{+}(u)$, we can show that $xv_{i} \in E(G)$. Suppose to the contrary that $xv_{i} \notin E(G)$. Then by Operation B, we have $cM_{2}(G+xv_{i}) > cM_{2}(G)$, a contradiction. Thus $xv_{i} \in E(G)$, that is $\overrightarrow{xv_{i}} \in A(F)$. Hence, for any $v_{i} \in N_{F}^{+}(u)$, $N_{F}^{-}(v_{i}) \supseteq N_{F}^{-}(u) \cup \{ u \}$ and we have $d_{F}^{-}(v_{i})>d_{F}^{-}(u)$.
\[
\begin{split}
& cM_{2}(G-E)-cM_{2}(G)\\
& \geq \sum_{v'=u, v_{1}, \ldots, v_{t}}(d_{F-\overrightarrow{E}}^{+}(v')-d_{F-\overrightarrow{E}}^{-}(v'))d_{F-\overrightarrow{E}}(v')^{2}\\
& ~~~-\sum_{v'=u, v_{1}, \ldots, v_{t}}(d_{F}^{+}(v')-d_{F}^{-}(v'))d_{F}(v')^{2}~~~\text{(by Observation~\ref{OB-2})}\\
& =-d_{F}^{-}(u)^{3}+\sum_{i=1}^{t}(d_{F}^{+}(v_{i})-(d_{F}^{-}(v_{i})-1)) (d_{F}(v_{i})-1)^{2}\\
& ~~~-(d_{F}^{+}(u)-d_{F}^{-}(u))d_{F}(u)^{2}-\sum_{i=1}^{t}(d_{F}^{+}(v_{i})-d_{F}^{-}(v_{i}))d_{F}(v_{i})^{2}\\
& =-d_{F}^{+}(u)^{3}+(d_{F}^{-}(u)-d_{F}^{+}(u))d_{F}^{+}(u)d_{F}^{-}(u) \\
& ~~~ + \sum_{i=1}^{t}((d_{F}(v_{i})-1)^{2}+ (d_{F}^{-}(v_{i})-d_{F}^{+}(v_{i}))(2d_{F}(v_{i})-1))\\
& ~~~\text{(since $d_{F}(u)=d_{F}^{+}(u)+d_{F}^{-}(u)$)}\\
& > -d_{F}^{+}(u)^{3}+\sum_{i=1}^{t}(d_{F}(v_{i})-1)^{2}~\text{(since $d_{F}^{-}(v_{i})>d_{F}^{+}(v_{i}) $, $d_{F}^{-}(u)>d_{F}^{+}(u) $)}\\
& \geq -d_{F}^{+}(u)^{3}+\sum_{i=1}^{t}d_{F}^{-}(u)^{2}~~~\text{ (since $d_{F}(v_{i}) \geq d_{F}^{-}(v_{i})>d_{F}^{-}(u)$  )}\\
& =-d_{F}^{+}(u)^{3}+d_{F}^{+}(u)d_{F}^{-}(u)^{2}\\
& >0~~~\text{(since $d_{F}^{-}(u)>d_{F}^{+}(u) $)},
\end{split}
\]
a contradiction.
\end{proof}

\begin{corollary}\label{cor-2}
For any $u \in X$ and $v \in Y$, if $uv \in E(G)$, then $\overrightarrow{uv} \in A(F)$.
\end{corollary}
\begin{proof}
Let $ u \in X$, $v \in Y$ with $uv \in E(G)$. By Claim~\ref{claim4}, $N_{F}^{-}(v) \subseteq X$; by Claim~\ref{claim3}, since $uv \in E(G)$, we have $\overrightarrow{uv} \in A(F)$.
\end{proof}

\begin{claim}\label{claim5}
For any $u, v \in X$, if $d_{G}(u) \geq d_{G}(v)$, then $N^{+}_{F}(u) \cap Y \supseteq N^{+}_{F}(v) \cap Y $.
\end{claim}
\begin{proof}
First, we show that $d_{F}^{-}(u) \leq d_{F}^{-}(v)$. For any $w \in N^{-}_{F[X]}(u)$, that is $w \in X$ and $\overrightarrow{wu} \in A(F)$, we have $d_{G}(w) > d_{G}(u)$; since $d_{G}(u) \geq d_{G}(v)$, $ d_{G}(w)>d_{G}(v)$. By Claim~\ref{claim1}, $wv \in E(G)$ and thus $\overrightarrow{wv} \in A(F)$, that is $w \in N_{F[X]}^{-}(v)$. Hence $N_{F[X]}^{-}(u) \subseteq N_{F[X]}^{-}(v)$. Similarly, $N_{F[X]}^{+}(u) \supseteq N_{F[X]}^{+}(v)$. By Corollary~\ref{cor-2}, $N_{F[X]}^{-}(u)=N_{F}^{-}(u)$ and $N_{F[X]}^{-}(v)=N_{F}^{-}(v)$. Hence, $N_{F}^{-}(u) \subseteq N_{F}^{-}(v)$ and thus $d_{F}^{-}(u) \leq d_{F}^{-}(v)$.

Next, we show that $d_{F}^{+}(u) \geq d_{F}^{+}(v)$.  For any $x \in X$, $d_{F}^{+}(x)=|N_{F}^{+}(x) \cap Y|+|N_{F[X]}^{+}(x)|$. Since $|N_{F[X]}^{+}(u)| \geq  |N_{F[X]}^{+}(v)|$, we only need to prove $|N_{F}^{+}(u) \cap Y| \geq |N_{F}^{+}(v) \cap Y|$.  Note that $|N_{F}^{+}(x) \cap Y|=d_{F}(x)-d_{F[X]}(x)$.

If $d_{G}(u)>d_{G}(v)$, since $|X|-2  \leq d_{G[X]}(u), d_{G[X]}(v) \leq |X|-1$ by Corollary~\ref{cor-1}, we have $d_{F}(u)-d_{F[X]}(u) \geq d_{F}(v)+1-(|X|-1) \geq d_{F}(v)-d_{F[X]}(v)  $. If $d_{G}(u)=d_{G}(v)$, then we can show that $d_{F[X]}(u)=d_{F[X]}(v)$. In fact, suppose to the contrary that $ d_{F[X]}(u)= |X|-2 $ and $d_{F[X]}(v)=|X|-1$ without loss of generality. Let $w$ be the only vertex in $X$ nonadjacent to $u$. Since $d_{F[X]}(v)=|X|-1$, $v \neq w$; by Claim~\ref{claim1}, $d_{G}(v) \neq d_{G}(u)$, a contradiction.

Finally, we show $N^{+}_{F}(u) \cap Y \supseteq N^{+}_{F}(v) \cap Y $. Suppose to the contrary that there exists $ w \in N_{F}^{+}(v) \cap Y \setminus N_{F}^{+}(u) \cap Y$. Then
\[
\begin{split}
& cM_{2}(G-vw+uw)-cM_{2}(G)\\
& \geq \sum_{v'=u, v} (d_{F-\overrightarrow{vw}+\overrightarrow{uw}}^{+}(v')-d_{F-\overrightarrow{vw}+\overrightarrow{uw}}^{-}(v'))
d_{F-\overrightarrow{vw}+\overrightarrow{uw}}(v')^{2}\\
& ~~~-\sum_{v'=u, v}(d_{F}^{+}(v')-d_{F}^{-}(v'))d_{F}(v')^{2}~~~\text{(by Observation~\ref{OB-2})}\\
& =(d_{F}^{+}(u)+1-d_{F}^{-}(u))(d_{F}(u)+1)^{2}+(d_{F}^{+}(v)-1-d_{F}^{-}(v))(d_{F}(v)-1)^{2}\\
& ~~~-(d_{F}^{+}(u)-d_{F}^{-}(u))d_{F}(u)^{2}-(d_{F}^{+}(v)-d_{F}^{-}(v))d_{F}(v)^{2}\\
& =(d_{F}(u)+1)^{2}-(d_{F}(v)-1)^{2}+(d_{F}^{+}(u)-d_{F}^{-}(u))(2d_{F}(u)+1)\\
& ~~~-(d_{F}^{+}(v)-d_{F}^{-}(v))(2d_{F}(v)-1)\\
& >(d_{F}^{+}(u)-d_{F}^{+}(v)+d_{F}^{-}(v)-d_{F}^{-}(u))(2d_{F}(v)-1)~~~\text{(since $d_{F}(u) \geq d_{F}(v)$)}\\
& \geq 0~~~\text{(since $d_{F}^{+}(u) \geq d_{F}^{+}(v)$ and $d_{F}^{-}(u) \leq d_{F}^{-}(v)$),}
\end{split}
\]
a contradiction.
\end{proof}

\begin{claim}\label{claim6}
For any $u, v \in X$, $N^{+}_{F}(u) \cap Y = N^{+}_{F}(v) \cap Y $
\end{claim}
\begin{proof}
Suppose to the contrary that there exist $u_{0}, v_{0} \in X$ such that $ N^{+}_{F}(u_{0}) \cap Y \neq N^{+}_{F}(v_{0}) \cap Y$. By Claim~\ref{claim5}, suppose that $d_{G}(u_{0})>d_{G}(v_{0})$ and thus $ N_{F}^{+}(u_{0}) \cap Y \supsetneq N_{F}^{+}(v_{0}) \cap Y$ without loss of generality. Let $w_{0} \in (N_{F}^{+}(u_{0}) \setminus N_{F}^{+}(v_{0})) \cap Y$.

For any $x \in N^{-}_{F}(w_{0}) $, since $w_{0} \in  (N_{F}^{+}(x) \setminus N_{F}^{+}(v_{0}))  \cap Y$, we have $ N_{F}^{+}(x) \cap Y \nsubseteq N_{F}^{+}(v_{0}) \cap Y$. By Claim~\ref{claim5}, we have $ N_{F}^{+}(x) \cap Y \supsetneq N_{F}^{+}(v_{0}) \cap Y $ and $d_{G}(x)>d_{G}(v_{0})$. By Claim~\ref{claim1}, $xv_{0} \in E(G)$ and thus $\overrightarrow{xv_{0}} \in A(F)$, that is $x \in N_{F}^{-}(v_{0})$. Hence, $N^{-}_{F}(w_{0}) \subseteq N^{-}_{F}(v_{0})$ and $ d_{F}^{-}(w_{0}) \leq d_{F}^{-}(v_{0}) $. Since $ v_{0} \in X$, we have $d_{F}^{+}(v_{0}) \geq d_{F}^{-}(v_{0}) $ and thus $d_{F}(v_{0}) \geq 2d_{F}^{-}(w_{0})$.
\[
\begin{split}
& cM_{2}(G+v_{0}w_{0})-cM_{2}(G)\\
& \geq \sum_{v'= v_{0}, w_{0}} (d_{F+ \overrightarrow{v_{0}w_{0}}}^{+}(v')-d_{F+ \overrightarrow{v_{0}w_{0}}}^{-}(v'))d_{F+ \overrightarrow{v_{0}w_{0}}}(v')^{2}\\
& ~~~-\sum_{v'=v_{0}, w_{0}} (d_{F}^{+}(v')-d_{F}^{-}(v'))d_{F}(v')^{2}~~~\text{(by Observation~\ref{OB-2})}\\
& =(d_{F}^{+}(v_{0})+1-d_{F}^{-}(v_{0}))(d_{F}(v_{0})+1)^{2}+(-d_{F}^{-}(w_{0})-1)(d_{F}^{-}(w_{0})+1)^{2}\\
& ~~~-( (d_{F}^{+}(v_{0})-d_{F}^{-}(v_{0}))d_{F}(v_{0})^{2}+(-d_{F}^{-}(w_{0})^{3}) )\\
& ~~~\text{(since $ d_{F}^{+}(w_{0})=0$ by Claim~\ref{claim4} and Corollary~\ref{cor-2})}\\
& \geq (d_{F}(v_{0})+1)^{2}-(3d_{F}^{-}(w_{0})^{2}+3d_{F}^{-}(w_{0})+1)\\
& ~~~\text{(since $d_{F}^{+}(v_{0})\geq d_{F}^{-}(v_{0})$ by $v_{0} \in X$)}\\
& \geq (2d_{F}^{-}(w_{0})+1)^{2}-(3d_{F}^{-}(w_{0})^{2}+3d^{-}_{F}(w_{0})+1)~\text{(since $ d_{F}(v_{0}) \geq 2d^{-}_{F}(w_{0})$)}\\
& =d_{F}^{-}(w_{0})^{2}+d_{F}^{-}(w_{0})\\
& >0,
\end{split}
\]
a contradiction.
\end{proof}

\begin{corollary}\label{cor-3}
For any $u \in X$, $N^{+}_{F}(u) \supseteq Y$.
\end{corollary}
\begin{proof}
For any $v \in Y$, we have $d_{F}^{-}(v) > d_{F}^{+}(v) \geq 0$; by Claim~\ref{claim4} and Corollary~\ref{cor-2}, there exists $w \in X$ such that $\overrightarrow{wv} \in A(F)$, that is $ v \in N_{F}^{+}(w) \cap Y$. By Claim~\ref{claim6}, we have $N_{F}^{+}(w) \cap Y=N_{F}^{+}(u) \cap Y$ and thus $ v \in N_{F}^{+}(u) \cap Y$. Hence, $Y \subseteq N_{F}^{+}(u) \cap Y$, that is $Y \subseteq N_{F}^{+}(u)$.
\end{proof}

By Corollary~\ref{cor-1}, denote $X_{1}:=\{u \in X: d_{G[X]}(u)=|X|-1 \}$ and $X_{2}:=\{u \in X: d_{G[X]}(u)=|X|-2 \}$. By Corollary~\ref{cor-3}, we have $d_{G}(u)=|X|+|Y|-1$ for $u \in X_{1}$ and $d_{G}(u)=|X|+|Y|-2$ for $u \in X_{2}$. Next, we will show that $X_{2} = \emptyset$ and thus $G$ is isomorphic to $K_{|X|} \vee \overline{K_{n-|X|}}$, completing the proof.

Suppose to the contrary that $X_{2} \neq \emptyset$. Let $u \in X_{2}$, that is $d_{G[X]}(u)=|X|-2$. Let $v \in X$ be the only nonadjacent vertex to $u$ in $G[X]$. Then $v \in X_{2}$ and by Corollary~\ref{cor-1}, for any $w \in X \setminus \{ u, v\}$, we have $d_{G}(w) \neq d_{G}(u)$ and thus $w \in X_{1}$. Hence, $X_{2}=\{u, v \}$ and $G[X]+uv$ is a complete graph. By Corollary~\ref{cor-3}, $N_{F}^{+}(u)=Y$. Since $u, v \in X$ and $uv \notin E(G)$, by Operation A, we have $d_{F}^{-}(u)=d_{F}^{+}(u)$, that is $ |X_{1}|=|Y|$. So $|X|=|Y|+2$. Note that $G+uv$ is isomorphic to $K_{|X|} \vee \overline{K_{|Y|}}$ and
\[
\begin{split}
& cM_{2}(K_{|X|} \vee \overline{K_{|Y|}} )-cM_{2}(G)\\
& = cM_{2}(G+uv) -cM_{2}(G) \\
& \geq cM_{2}(F+uv)-cM_{2}(F)\\
& =\sum_{v'=u,v}(d_{F+uv}^{+}(v')-d_{F+uv}^{-}(v'))(d_{F}(v'))^{2}-\sum_{v'=u,v}(d_{F}^{+}(v')-d_{F}^{-}(v'))(d_{F}(v'))^{2}\\
& =0.
\end{split}
\]
Since $|X|=|Y|+2$, we have
$cM_{2}(K_{|Y|} \vee \overline{K_{|X|}})=|X||Y|((n-1)^{2}-|Y|^{2})>cM_{2}(K_{|X|} \vee \overline{K_{|Y|}})=|X||Y|((n-1)^{2}-|X|^{2})$,
a contradiction.

\acknowledgment{This work was supported by the National Natural Science Foundation of China (No. 12201623) and the Natural Science Foundation of Jiangsu Province (No. BK20221105)}

\singlespacing

\end{document}